\documentclass[]{article}

\usepackage{graphicx}
\usepackage{tikz}
\usepackage{amsmath}
\usepackage{amssymb}
\usepackage{amsthm}
\usepackage{mathrsfs}
\IfFileExists{stmaryrd.sty}{%
  \usepackage{stmaryrd}
}{%

}
\usepackage[T1]{fontenc}
\usepackage{pxfonts}
\usepackage{enumerate}
\usepackage{color}
\usepackage{enumitem}
\usepackage[hidelinks]{hyperref}

\allowdisplaybreaks
\makeatletter
\renewcommand\section{\@startsection{section}{1}{\z@}%
  {-3.5ex \@plus -1ex \@minus -.2ex}%
  {2.3ex \@plus .2ex}%
  {\normalfont\bfseries\scshape\centering\fontsize{11}{14}\selectfont}}
\renewcommand\subsection{\@startsection{subsection}{2}{\z@}%
  {-3.25ex \@plus -1ex \@minus -.2ex}%
  {1.5ex \@plus .2ex}%
  {\normalfont\bfseries\scshape\fontsize{11}{14}\selectfont}}
\makeatother

\usepackage{fancyhdr}
\usepackage[nottoc,notlot,notlof]{tocbibind}

\newcommand\shorttitle{Heat flow conjecture for random matrices}
\newcommand\authors{T. Assiotis}

\newtheorem{thm}{Theorem}[section]
\newtheorem{cor}[thm]{Corollary}
\newtheorem{lem}[thm]{Lemma}
\newtheorem{prop}[thm]{Proposition}

\newcommand{\E}{\mathbb{E}}
\newcommand{\Prob}{\mathbb{P}}
\newcommand{\R}{\mathbb{R}}
\newcommand{\C}{\mathbb{C}}
\newcommand{\D}{\mathbb{D}}
\newcommand{\tr}{\operatorname{tr}}
\newcommand{\supp}{\operatorname{supp}}

\newcommand{\dd}{\mathrm{d}}
\newcommand{\defeq}{\overset{\mathrm{def}}{=}}
\definecolor{PlanarBlue}{rgb}{0.12,0.35,0.60}
\definecolor{BoundaryTeal}{rgb}{0.08,0.43,0.39}

\title{\large\bf ON THE HEAT FLOW CONJECTURE FOR RANDOM MATRICES}
\author{\small THEODOROS ASSIOTIS}
\date{}

\begin{document}

\maketitle

\begin{abstract}
We prove two families of distinguished cases of the heat flow conjecture for random
matrices of Hall--Ho: elliptic Gaussian sources with planar targets, and
deterministic Hermitian sources with targets on the boundary of the
covariance disk.  In particular, the
empirical zero measure of the heat-evolved characteristic polynomial of a
complex Ginibre matrix converges almost surely to the semicircle law. The strategy of the proof is to transfer estimates for the corresponding well-understood Gaussian
matrix ensembles, through exact heat identities, into control of the zeros
of the heat-evolved characteristic polynomials.
\end{abstract}

\section{Introduction}\label{SectionIntroduction}

\subsection{Setting and main results}\label{SectionMainResults}

The remarkable heat flow conjecture of Hall and Ho
\cite[Conjecture~2.3]{HallHo} relates polynomial heat flow to a change
in the covariance of a Gaussian matrix perturbation. We prove two
families of distinguished cases of the conjecture: elliptic Gaussian sources with planar targets, and
bounded deterministic Hermitian sources with targets on the boundary
of the covariance disk. The latter include the convergence from the
characteristic polynomial of a complex Ginibre matrix to the semicircle law. Let us make our results precise. A more detailed literature review will then follow.

For \(s>0\) and \(c\in\C\) with \(|c|\leq s\), let
\(\mathbf Z_N(s,c)\) be the centred jointly real-Gaussian matrix with
entry covariances
\begin{equation}\label{EqHallCovariance}
 \E\left[\left(\mathbf Z_N\left(s,c\right)\right)_{ij}
       \overline{\left(\mathbf Z_N\left(s,c\right)\right)_{kl}}\right]
 =\frac{s}{N}\delta_{ik}\delta_{jl},
 \qquad
 \E\left[\left(\mathbf Z_N\left(s,c\right)\right)_{ij}\left(\mathbf Z_N\left(s,c\right)\right)_{kl}\right]
 =\frac{c}{N}\delta_{il}\delta_{jk}.
\end{equation}
Here all real and imaginary parts together form a real Gaussian vector.
The covariance condition is precisely \(|c|\leq s\). At \(c=0\) we
obtain complex Ginibre matrices with entry variance \(s/N\), while
\(c=s\) gives the Gaussian unitary ensemble (GUE) of variance \(s\).
Our convention is related to that of \cite{HallHo} by \(c=s-\tau\).

For a polynomial \(P\) and \(t\in\C\), define the terminating heat series
\begin{equation}\label{EqPolynomialHeatFlowDefinition}
 \exp\!\left(\frac{t}{2N}\partial_z^2\right)P
 \defeq\sum_{k=0}^{\left\lfloor\deg\left(P\right)/2\right\rfloor}
   \frac1{k!}\left(\frac{t}{2N}\right)^k\partial_z^{2k}P.
\end{equation}
Changing the covariance parameter from \(c_0\) to \(c_1\) corresponds to
\begin{equation}\label{EqHeatSign}
 q_N\defeq\exp\!\left(-\frac{c_1-c_0}{2N}\partial_z^2\right)p_N.
\end{equation}
Informally, the Hall--Ho conjecture predicts that the zeros obtained by
applying this operator to a perturbed characteristic polynomial have
the same limiting distribution as the eigenvalues obtained by changing
the matrix perturbation from \(\mathbf Z_N(s,c_0)\) to
\(\mathbf Z_N(s,c_1)\). The original formulation imposes abstract hypotheses on
the limiting noncommutative distribution of the source; we refer to
\cite{HallHo} for that formulation. Our two results below state directly
the source assumptions needed here (which can be shown to satisfy the abstract hypotheses alluded to above from \cite{HallHo} but we will not do this explicitly here; this is simply to avoid introducing abstract preliminaries which are not the main point of the paper).

If \(P\) is monic of degree \(N\), and \(\mathbf M\) is an
\(N\times N\) matrix, write
\[
 \nu_P\defeq\frac1N\sum_{P\left(\zeta\right)=0}\delta_\zeta,
 \qquad
 L_{\mathbf M}\defeq\frac1N\sum_{j=1}^N\delta_{\lambda_j\left(\mathbf M\right)}.
\]
Both sums count multiplicities. We write
\(\rho_N\overset{\mathrm w}{\longrightarrow}\rho\) for weak convergence
against bounded continuous functions, and \(T_*\rho\) for pushforward
under \(T\). We use the normalized trace
\(\tr_N=N^{-1}\operatorname{Tr}\), the operator norm \(\|\cdot\|\), and the
Hilbert--Schmidt norm
\(\|\mathbf M\|_{\mathrm{HS}}^2=\operatorname{Tr}(\mathbf M\mathbf M^*)\).
Write \(\dd A(z)=\dd x\,\dd y\) for planar Lebesgue measure,
\(\D=\{z\in\C:|z|<1\}\), and
\(\mathscr D(z,r)=\{w\in\C:|w-z|<r\}\).
For \(S>0\) and \(|c|<S\), define the ellipse and its uniform probability measure by
\begin{equation}\label{EqEllipseDefinition}
 \mathscr E_{S,c}\defeq
 \left\{\sqrt S\,u+\frac{c}{\sqrt S}\bar u:u\in\C,\ \left|u\right|\leq1\right\},
\end{equation}
\begin{equation}\label{EqEllipseDensity}
 \dd\mu_{S,c}\left(z\right)\defeq\frac{S}{\pi\left(S^2-\left|c\right|^2\right)}
    \mathbf1_{\mathscr E_{S,c}}\left(z\right)\dd A\left(z\right).
\end{equation}
For \(c=|c|\mathrm{e}^{2\mathrm{i}\theta}\), the semiaxes are
\(\sqrt S\pm|c|/\sqrt S\), rotated through \(\theta\);
see Figure~\ref{FigTargetEllipse}.

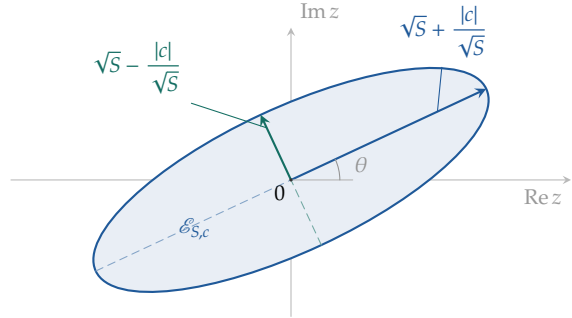
\begin{figure}[!ht]
\centering
\begingroup
\def\EllipseVariance{1}
\def\EllipseModulus{0.5}
\def\EllipseAngle{25}
\pgfmathsetmacro{\EllipseMajor}{sqrt(\EllipseVariance)+\EllipseModulus/sqrt(\EllipseVariance)}
\pgfmathsetmacro{\EllipseMinor}{sqrt(\EllipseVariance)-\EllipseModulus/sqrt(\EllipseVariance)}
\begin{tikzpicture}[x=1.9cm,y=1.9cm,>=stealth,
  line cap=round,line join=round,font=\footnotesize]
  \path[use as bounding box] (-2.15,-1.02) rectangle (2.15,1.18);
  \draw[gray!55,thin,->] (-1.95,0) -- (1.95,0)
    node[below left,text=gray!80] {\(\operatorname{Re}z\)};
  \draw[gray!55,thin,->] (0,-0.95) -- (0,1.06)
    node[above right,text=gray!80] {\(\operatorname{Im}z\)};
  \begin{scope}[rotate=\EllipseAngle]
    \filldraw[fill=PlanarBlue!10,draw=PlanarBlue,line width=0.8pt]
      (0,0) ellipse[x radius=\EllipseMajor,y radius=\EllipseMinor];
    \draw[PlanarBlue!55,densely dashed,thin] (-\EllipseMajor,0) -- (0,0);
    \draw[BoundaryTeal!60,densely dashed,thin] (0,-\EllipseMinor) -- (0,0);
    \draw[PlanarBlue,line width=0.8pt,->] (0,0) -- (\EllipseMajor,0);
    \draw[BoundaryTeal,line width=0.8pt,->] (0,0) -- (0,\EllipseMinor);
    \coordinate (majorlabelpoint) at ({0.75*\EllipseMajor},0);
    \coordinate (minorlabelpoint) at (0,{0.8*\EllipseMinor});
  \end{scope}
  \draw[gray!60,thin] (0,0) -- (0.43,0);
  \draw[gray!85,thin] (0.34,0)
    arc[start angle=0,end angle=\EllipseAngle,radius=0.34];
  \node[text=gray!85] at ({0.5*cos(\EllipseAngle/2)},
    {0.5*sin(\EllipseAngle/2)}) {\(\theta\)};
  \fill[black!80] (0,0) circle[radius=0.6pt];
  \node[below left,inner sep=2pt] at (0,0) {\(0\)};
  \node[text=PlanarBlue] (majorlabel) at (1.05,1.03)
    {\(\sqrt S+\dfrac{|c|}{\sqrt S}\)};
  \draw[PlanarBlue,thin] (majorlabel.south) -- (majorlabelpoint);
  \node[text=BoundaryTeal] (minorlabel) at (-1.08,0.78)
    {\(\sqrt S-\dfrac{|c|}{\sqrt S}\)};
  \draw[BoundaryTeal,thin] (minorlabel.south east) -- (minorlabelpoint);
  \node[text=PlanarBlue] at (-0.67,-0.34) {\(\mathscr E_{S,c}\)};
\end{tikzpicture}
\endgroup
\caption{The ellipse \(\mathscr E_{S,c}\), with semiaxes
\(\sqrt S\pm|c|/\sqrt S\) and rotation angle \(\theta=\arg(c)/2\).}
\label{FigTargetEllipse}
\end{figure}

Our first theorem identifies the heat-evolved zeros for planar targets.
It allows a convergent scalar source together with an independent
elliptic Gaussian source, which need not be normal. We refer to Section \ref{SectionSimulations} for figures of simulations illustrating the result.

\begin{thm}\label{ThmPlanar}
Fix \(s>0\), \(s_0\geq0\), and \(d,\kappa_0,\kappa_1\in\C\) satisfying
\[
 \left|d\right|\leq s_0,
 \qquad
 \left|\kappa_j\right|\leq s\quad \left(j=0,1\right).
\]
Let \((\mathbf{Y}_N)_{N\geq1}\) be a sequence of centred, jointly
real-Gaussian matrices satisfying
\[
 \E\left[\left(\mathbf{Y}_N\right)_{ij}\overline{\left(\mathbf{Y}_N\right)_{kl}}\right]
 =\frac{s_0}{N}\delta_{ik}\delta_{jl},
 \qquad
 \E\left[\left(\mathbf{Y}_N\right)_{ij}\left(\mathbf{Y}_N\right)_{kl}\right]
 =\frac{d}{N}\delta_{il}\delta_{jk}.
\]
For each \(j=0,1\), let
\((\mathbf Z_N(s,\kappa_j))_{N\geq1}\) be a sequence of centred elliptic
Gaussian matrices with covariances \eqref{EqHallCovariance} at \(c=\kappa_j\).  Assume that the
sigma-algebras
\[
 \sigma\left(\mathbf{Y}_N:N\geq1\right)
 \quad\text{and}\quad
 \sigma\!\left(
 \mathbf Z_N\left(s,\kappa_j\right):N\geq1,\ j\in\left\{0,1\right\}
 \right)
\]
are independent.  Put
\begin{equation}\label{EqTotalCovariances}
 S\defeq s_0+s,
 \qquad
 c_j\defeq d+\kappa_j\quad \left(j=0,1\right),
\end{equation}
and assume \(\left|c_0\right|,\left|c_1\right|<S\).  If
\((b_N)\) is a deterministic scalar sequence with \(b_N\to b\in\C\), set
\[
 p_N\left(z\right)\defeq\det\!\left(z\mathbf{I}-b_N\mathbf{I}-\mathbf{Y}_N
                         -\mathbf Z_N\left(s,\kappa_0\right)\right),
 \qquad
 q_N\defeq\exp\!\left(-\frac{c_1-c_0}{2N}\partial_z^2\right)p_N.
\]
Then, for every joint coupling satisfying the preceding covariance and
independence assumptions,
\begin{equation}\label{EqPlanarJointLimit}
 \begin{aligned}
 \nu_{q_N}
 &\overset{\mathrm{w}}{\longrightarrow}
 \left(z\mapsto b+z\right)_*\mu_{S,c_1},\\
 L_{\,b_N\mathbf{I}+\mathbf{Y}_N+\mathbf Z_N\left(s,\kappa_1\right)}
 &\overset{\mathrm{w}}{\longrightarrow}
 \left(z\mapsto b+z\right)_*\mu_{S,c_1},
 \end{aligned}
 \qquad\text{almost surely}.
\end{equation}
\end{thm}

The target-matrix convergence in \eqref{EqPlanarJointLimit} is the
classical elliptic law \cite[Theorem~1.8]{NguyenORourke}; we recall the
arbitrary-coupling version that we need in Lemma~\ref{LemAlmostSureEllipticLaw}.
The actual novel assertion proved here is the convergence of the heat-evolved zeros.
Taking \(s_0=d=0\) gives the zero-source and scalar-source cases
throughout the open covariance disk. For a nonzero Gaussian source,
the total covariances \(S\) and \(c_j\) absorb its contribution.

Moving on, to state the boundary case result, we introduce additive convolution through
analytic transforms. For a compactly supported probability measure
\(\rho\) on \(\R\), its Cauchy transform is
\begin{equation}\label{EqCauchyTransformDefinition}
 G_\rho\left(z\right)\defeq\int_{\R}\frac{\dd\rho\left(x\right)}{z-x},
 \qquad z\in\C\setminus\supp\rho.
\end{equation}
Since \(G_\rho(z)=z^{-1}+O(z^{-2})\) at infinity, it has an inverse
branch \(K_\rho\) for small nonzero \(w\), characterized by
\[
 G_\rho\left(K_\rho\left(w\right)\right)=w,
 \qquad K_\rho\left(w\right)=\frac1w+O\left(1\right).
\]
The \(R\)-transform
\begin{equation}\label{EqRTransformDefinition}
 R_\rho\left(w\right)\defeq K_\rho\left(w\right)-\frac1w
\end{equation}
extends analytically to zero. For compactly supported probabilities
\(\rho,\eta\) on \(\R\), their additive convolution
\(\rho\boxplus\eta\) is the unique compactly supported probability satisfying
\begin{equation}\label{EqAdditiveConvolutionDefinition}
 R_{\rho\boxplus\eta}\left(w\right)=R_\rho\left(w\right)+R_\eta\left(w\right)
 \qquad\text{for \(w\) sufficiently close to zero},
\end{equation}
see
\cite[Theorems~4.5 and~4.12]{SpeicherLectures}.We shall
use it only with the centred semicircle law of variance \(s\), given by
\[
 \dd\mathfrak{sc}_s\left(x\right)\defeq
 \frac{\sqrt{4s-x^2}}{2\pi s}
 \mathbf1_{\left[-2\sqrt s,2\sqrt s\right]}\left(x\right)\dd x,
 \qquad R_{\mathfrak{sc}_s}\left(w\right)=sw.
\]

Our second theorem identifies the zero limit when the target Gaussian
perturbation is aligned with the affine line containing the source spectrum.
The limit is the additive convolution just defined,
followed by the corresponding rotation and translation. We refer to Section \ref{SectionSimulations} for figures of simulations illustrating the result.

\begin{thm}\label{ThmBoundary}
Fix \(s>0\), \(u\in\C\) with \(\left|u\right|=1\), and
\(c_0\in\C\) with \(\left|c_0\right|\leq s\).  Put \(c_1=su^2\).
Let \(\mathbf{A}_N=\mathbf{A}_N^*\) be deterministic and satisfy
\[
 \sup_N\left\|\mathbf{A}_N\right\|<\infty,
 \qquad
 L_{\mathbf{A}_N}\overset{\mathrm{w}}{\longrightarrow}\mu_0,
\]
and let \((b_N)\) be a deterministic scalar sequence with
\(b_N\to b\in\C\).  Let
all the matrices \(\mathbf{Z}_N(s,c_j)\), \(N\geq1\), \(j=0,1\), be
defined on a common probability space.  Assume that, for each \(N\) and
\(j\), the entries of \(\mathbf{Z}_N(s,c_j)\) are centred and jointly
real-Gaussian, and satisfy
\[
 \E\left[\left(\mathbf{Z}_N\left(s,c_j\right)\right)_{ab}
       \overline{\left(\mathbf{Z}_N\left(s,c_j\right)\right)_{kl}}\right]
 =\frac{s}{N}\delta_{ak}\delta_{bl},
 \qquad
 \E\left[\left(\mathbf{Z}_N\left(s,c_j\right)\right)_{ab}
       \left(\mathbf{Z}_N\left(s,c_j\right)\right)_{kl}\right]
 =\frac{c_j}{N}\delta_{al}\delta_{bk}.
\]
The matrices may be coupled arbitrarily with each other and across \(N\).
Define
\[
 p_N\left(z\right)\defeq\det\!\left(z\mathbf{I}-b_N\mathbf{I}
              -u\mathbf{A}_N-\mathbf{Z}_N\left(s,c_0\right)\right),
 \qquad
 q_N\defeq\exp\!\left(-\frac{c_1-c_0}{2N}\partial_z^2\right)p_N.
\]
Then, almost surely,
\begin{equation}\label{EqBoundaryLimit}
 \begin{aligned}
 \nu_{q_N}
 &\overset{\mathrm{w}}{\longrightarrow}
 \left(x\mapsto b+ux\right)_*\left(\mu_0\boxplus\mathfrak{sc}_s\right),\\
 L_{\,b_N\mathbf{I}+u\mathbf{A}_N+\mathbf{Z}_N\left(s,c_1\right)}
 &\overset{\mathrm{w}}{\longrightarrow}
 \left(x\mapsto b+ux\right)_*\left(\mu_0\boxplus\mathfrak{sc}_s\right).
 \end{aligned}
\end{equation}
\end{thm}

The condition \(c_1=su^2\) makes \(u^{-1}\mathbf Z_N(s,c_1)\)
a GUE matrix of variance \(s\). Thus the spectra of the source and target matrices lie
on the affine line \(b_N+u\R\). The target-matrix convergence in
\eqref{EqBoundaryLimit} is the classical deformed-GUE law
\cite{CapitaineDonatiMartinFeralFevrier}, recalled in
Lemma~\ref{LemDeformedGUE} while the genuinely novel assertion proven here is the heat-evolved zero convergence.
The next corollary records the Ginibre-to-semicircle case singled out
in \cite[Example~2.5]{HallHo}; it follows by taking a zero source and
the real boundary target, see again Section \ref{SectionSimulations} for figures of simulations illustrating the result.

\begin{cor}\label{CorGinibreSemicircle}
Let \((\mathbf{G}_N)_{N\geq1}\) be any sequence of complex Ginibre
matrices: for each \(N\), its entries are independent centred complex
Gaussian variables satisfying
\[
 \E\!\left[\left|\left(\mathbf{G}_N\right)_{ij}\right|^2\right]=\frac1N,
 \qquad
 \E\!\left[\left(\mathbf{G}_N\right)_{ij}^2\right]=0.
\]
Let us set
\[
 p_N\left(z\right)\defeq\det\left(z\mathbf{I}-\mathbf{G}_N\right),\qquad
 q_N\left(z\right)\defeq\exp\!\left(-\frac{1}{2N}\partial_z^2\right)p_N\left(z\right).
\]
Then, almost surely,
 $\nu_{q_N}\overset{\mathrm{w}}{\longrightarrow}\mathfrak{sc}_1$.
\end{cor}

 Finally, we note that there is a question of universality of the deformation phenomenon obtained via heat flow as studied here, see \cite{HallHo} for further interesting discussion on this phenomenon. The specifics of the proof do not extend in any obvious way, as far as I can tell, to other settings but the meta-idea of comparing with a tractable ensemble of random matrices may be useful. It would be interesting to have a more conceptual understanding of this.

\subsection{Relation to previous work}

We now review the relevant literature which in the past 5-10 years has witnessed intense growth.
Polynomial heat flow has been studied through its interacting zeros
\cite{TaoHeatFlow,TaoHeatFlowCircle}, and also occurs as a tool in interacting particle
systems \cite{AssiotisInteracting,NicaQuastelRemenik}.
Several recent works establish limiting zero evolutions for other
input classes. Hall, Ho, Jalowy and Kabluchko treat the planar Gaussian
analytic function \cite{HallHoJalowyKabluchkoGAF}, polynomials with
independent coefficients \cite{HallHoJalowyKabluchkoZeros}, and repeated
differential operators \cite{HallHoJalowyKabluchkoDifferentiation}.
Hadamard products, finite free convolutions and repeated differential
operators are also studied through coefficient profiles by Jalowy,
Kabluchko and Marynych \cite{JalowyKabluchkoMarynychProfiles}.
Kabluchko's backward heat-flow result \cite{KabluchkoHeat} and the
operator framework of Campbell and Jalowy \cite{CampbellJalowy} use
real-rooted input polynomials. More recently, H{\"o}fert, Jalowy and Kabluchko
\cite{HofertJalowyKabluchkoHeat} allow complex roots for increasing powers
of one fixed polynomial. The present results concern characteristic polynomials with strongly
dependent coefficients, arising from elliptic Gaussian matrices and
deterministic Hermitian sources with elliptic Gaussian perturbations which do not fall into the scope of the above works.

Repeated differentiation provides another setting in which an operation
on polynomials induces an evolution of their zero distributions.
Steinerberger \cite{SteinerbergerTransport} proposed a nonlocal transport
equation for the density of real roots, and O'Rourke and Steinerberger
\cite{ORourkeSteinerbergerTransport} developed a corresponding model
for rotationally invariant distributions in the complex plane.
Galligo \cite{GalligoComplexMotion} proposed a coupled system for
anisotropic complex-root motion. Rigorous limiting results for a number
of derivatives proportional to the degree were obtained by Feng and Yao
\cite{FengYaoDerivatives} for polynomials with independent coefficients
and by Hoskins and Kabluchko \cite{HoskinsKabluchkoDerivatives} for
real-rooted polynomials. In the latter setting, the limiting law is a
rescaled fractional free convolution power.

Randomized differentiation gives a further connection with random
matrices. Galligo, Najnudel and Vu \cite{GalligoNajnudelVuRandomized}
study randomized derivative operations and their effect on limiting zero
distributions. For real roots in a fixed bounded interval, Galligo and
Najnudel \cite{GalligoNajnudelDynamics} show that a family of such
operations, related to matrix minors, has the same macroscopic limit as
ordinary differentiation. For complex roots, Galligo, Najnudel and Vu
\cite{GalligoNajnudelVuRadial} prove a radial limiting law for a
structured deterministic sampling on concentric circles. Najnudel and Vu
\cite{NajnudelVuRadial} subsequently weaken the sampling assumptions,
while Hall and Perales \cite{HallPeralesRadial} give a simpler proof and
extend the analysis to further differential operators.

Coming back to our immediate setting the proofs below combine certain established Gaussian identities and some potential-theoretic
methods. Let us comment on the relevant inputs from the literature. Hall and Ho prove the matrix second-moment identity
\cite[Theorem~2.7]{HallHo} and convergence of holomorphic moments under
their hypotheses \cite[Theorem~5.1]{HallHo}. Holomorphic moments alone do
not determine a planar probability measure. The scalar norm identity is
a Euclidean specialization of Driver, Hall and Kemp's complex-time
Segal--Bargmann transform \cite{DriverHallKemp}; the characteristic-polynomial
kernel is a case of Akemann and Vernizzi's formula \cite{AkemannVernizzi}.
Our planar argument uses the exponential rates of shifted Gaussian norms
to identify the potential inside the target ellipse. It is related to
the potential-identification methods of Bayraktar \cite{BayraktarMass}
and Bloom--Levenberg \cite{BloomLevenberg}, with a family of shifted
norm rates as its input. Finally, at the boundary, we adapt the regularized-log
determinant estimate of Ben Arous, Bourgade and McKenna
\cite{BenArousBourgadeMcKenna} and use an interval version of weighted
polynomial extremality \cite{MhaskarSaff,BloomWeighted}.

\subsection{Strategy of proof}\label{SectionProofOutline}

The matrix limits in Theorems~\ref{ThmPlanar} and~\ref{ThmBoundary}
are classical. Our task is then to transfer information about these Gaussian
ensembles to the zeros of the heat-evolved characteristic polynomial.
The two exact identities used for this transfer are basically known:
the scalar Gaussian norm identity of Driver--Hall--Kemp
(Proposition~\ref{PropNormTransfer}) and the matrix second-moment
identity of Hall--Ho (Proposition~\ref{PropMatrixDeformation}).
Figure~\ref{FigProofStrategy} shows what each identity transfers and
how the transferred estimates determine the zero measure.

Weak convergence of the matrix eigenvalue measures alone does not
provide these estimates. The logarithmic kernel is singular, and
expected squared determinants involve an exponential of its integral.
For the planar case we need matching exponential rates for a family of shifted
Gaussian norms, together with a global growth bound while for the boundary case
a uniform upper bound for an expected squared determinant suffices,
because monicity and weighted polynomial extremality supply the
matching lower bound for the weighted polynomial norm.

\textbf{The planar case.}
After adding Gaussian covariances, scaling and translation, the initial
polynomial is \(p_N(z)=\det(z\mathbf I-\mathbf X_{N,c_0})\), where
\(\mathbf X_{N,c_0}\) is elliptic Gaussian with total variance one.
The target parameter satisfies \(|c_1|<1\), and
\(q_N=\exp(-(c_1-c_0)\partial_z^2/(2N))p_N\).
For \(v>\max\{|c_0|,|c_1|\}\), let \(\gamma_{N;v,c,a}\) be the complex
Gaussian law with mean \(a\) and centered moments
\(\E[|Z-a|^2]=v/N\), \(\E[(Z-a)^2]=c/N\).
Proposition~\ref{PropNormTransfer}, the polynomial specialization of
\cite[Theorem~1.6 and Section~1.4]{DriverHallKemp}, gives the exact identity
\[
 \int_{\C}\left|q_N\left(z\right)\right|^2
       \dd\gamma_{N;v,c_1,a}\left(z\right)
 =\int_{\C}\left|p_N\left(z\right)\right|^2
       \dd\gamma_{N;v,c_0,a}\left(z\right).
\]
It holds for every realization and every shift \(a\).

Choose \(\max\{|c_0|,|c_1|\}<v<1\) and shifts \(a=(1-v)u\), with
\(u\in\D\). The initial elliptic law supplies a lower bound for the
norm on the right. The characteristic-polynomial
kernel and Mehler's formula supply an expected upper bound, which
Markov's inequality and Borel--Cantelli turn into the matching
almost-sure bound. Proposition~\ref{PropShiftedExponent} thus gives
\[
 \begin{gathered}
 J_N\left(u\right)\defeq
 \int_{\C}\left|p_N\left(z\right)\right|^2
       \dd\gamma_{N;v,c_0,\left(1-v\right)u}\left(z\right),\\
 \frac1N\log J_N\left(u\right)\longrightarrow M\left(u\right),
 \qquad M\left(u\right)\defeq-1+\left(1-v\right)\left|u\right|^2,
 \end{gathered}
\]
simultaneously on any prescribed countable dense subset of \(\D\).
The exact identity transfers these rates to the corresponding norms
of \(q_N\).

To see why this family of estimates determines the zeros, write the
Gaussian density as a constant times \(\mathrm{e}^{-NQ_{v,c,a}}\),
with \(Q_{v,c,a}(a)=0\), and put
\(h_N=N^{-1}\log|q_N|\), \(V_c=Q_{1,c,0}/2-1/2\), and
\(L_c(u)=u+c\overline u\).
The target ellipse is \(L_{c_1}(\overline\D)\), and \(V_{c_1}\) is
the potential of its uniform law in the interior
(Lemma~\ref{LemEllipticPotential}).
The same identity at \(v=1\), \(a=0\), and the characteristic-polynomial
kernel compute the expected unshifted squared norm of \(q_N\).
Markov--Borel--Cantelli, Stirling's formula and the evaluation bound
of Lemma~\ref{LemEvaluationKernel} then give
\(h_N\leq V_{c_1}+\varepsilon_N\) almost surely, where
\(\varepsilon_N\to0\) is independent of \(z\).
This global bound is needed both for subharmonic compactness and to
control the Gaussian tails.
Monicity excludes collapse to \(-\infty\), so subsequential local
\(\mathrm L^1\) limits \(\psi\leq V_{c_1}\) exist.

For each \(u\) in the chosen countable dense subset of \(\D\),
the Gaussian shift \(a=(1-v)u\) gives, by quadratic completion,
\[
 2V_{c_1}\left(z\right)-Q_{v,c_1,\left(1-v\right)u}\left(z\right)
 =M\left(u\right)-B_{v,c_1}\left(z-L_{c_1}\left(u\right)\right),
\]
where \(B_{v,c_1}\) is positive definite. The transferred rate, so-called Hartogs'
lemma and the tail bound force \(\psi\) to attain \(V_{c_1}\) at the
unique maximizing point \(L_{c_1}(u)\). The dense family of shifts
therefore gives equality throughout the ellipse interior. Here, with \(Q=Q_{v,c_1,(1-v)u}\), the matching lower bound on the
shifted norm is essential: together with Hartogs' lemma and the tail
estimate, it forces \(\sup_{\mathbb{C}}(2\psi-Q)=M(u)\).
Since \(2\psi-Q\leq2V_{c_1}-Q\), whose unique maximum is \(M(u)\)
at \(L_{c_1}(u)\), this yields
\(\psi(L_{c_1}(u))=V_{c_1}(L_{c_1}(u))\).
Since \(\nu_{q_N}=(2\pi)^{-1}\Delta h_N\), the limiting zero measure
has the uniform ellipse density there. Its total mass is at most one,
and this interior already carries mass one, proving
Proposition~\ref{PropNormalisedPlanar}.
The target matrix's elliptic law is used separately to obtain the
matching matrix limit in Theorem~\ref{ThmPlanar} completing the proof.

\textbf{The boundary case.}
After rotation and translation, the target is
\(\mathbf B_N=\mathbf A_N+\mathbf H_N\), with \(\mathbf A_N\) a uniformly
bounded deterministic Hermitian source with limiting law \(\mu_0\),
and \(\mathbf H_N\) GUE of variance \(s\).
The initial polynomial \(p_N\) is the characteristic polynomial of
\(\mathbf A_N+\mathbf Z_N(s,c_0)\), where \(|c_0|\leq s\), and
\(q_N=\exp(-(s-c_0)\partial_z^2/(2N))p_N\).
Proposition~\ref{PropMatrixDeformation}, namely
\cite[Theorem~2.7, Equation~(2.11)]{HallHo}, gives
\[
 \E\left[\left|q_N\left(z\right)\right|^2\right]
 =\E\left[\left|\det\left(z\mathbf I-\mathbf B_N\right)\right|^2\right].
\]
The proof then uses information about this target ensemble only and in particular it requires
no limiting eigenvalue law for the initial nonnormal matrix.

Put \(\mu=\mu_0\boxplus\mathfrak{sc}_s\), the classical deformed-GUE
limit, and \(\overline\mu_N=\E[L_{\mathbf B_N}]\).
Lemma~\ref{LemDeformedGUE} gives that the Wasserstein distance satisfies
\(\mathbb W_1(\overline\mu_N,\mu)\to0\), compact support and a bounded
density for \(\mu\), hence continuity of
\(U_\mu(z)=\int_\R\log|z-x|\dd\mu(x)\).
These facts and Gaussian concentration yield the stronger estimate
in Lemma~\ref{LemDeterminantBound}:
\begin{equation}\label{EqDetBoundIntro}
 \E\left[\left|\det\left(x\mathbf I-\mathbf B_N\right)\right|^2\right]
 \leq\mathrm{e}^{2N\left[U_\mu\left(x\right)+o\left(1\right)\right]},
\end{equation}
uniformly on a compact interval \(K\) containing \(\supp\mu\).
To handle the logarithmic singularity, fix \(T>0\) and replace
\(\log|x-y|\) by \(\max\{\log|x-y|,-T\}\). Gaussian concentration
bounds the exponential moment of the resulting matrix statistic in
terms of its mean. Wasserstein convergence then compares this mean
with the corresponding integral against \(\mu\), while the bounded
density of \(\mu\) controls the error introduced by truncation.
More precisely, uniformly for \(x\in K\),
\[
\begin{aligned}
 &\frac1{2N}\log\E\left[\left|\det\left(x\mathbf I-\mathbf B_N\right)\right|^2\right]
       -U_\mu\left(x\right)\leq
 \mathrm{e}^T\mathbb W_1\left(\overline\mu_N,\mu\right)
 +2C_s\mathrm{e}^{-T}+\frac{s\mathrm{e}^{2T}}N,
\end{aligned}
\]
where \(C_s\) bounds the density of \(\mu\). For fixed \(T\), the
first and third terms vanish as \(N\to\infty\); subsequently letting
\(T\to\infty\) removes the remaining truncation error, giving \eqref{EqDetBoundIntro}.

The exact second-moment identity transfers this uniform estimate to
\(q_N\). Uniformity permits integration against the probability
arcsine measure \(\alpha_K\). Markov's inequality and Borel--Cantelli
then give
\[
 I_N\defeq\int_K\left|q_N\left(x\right)\right|^2
          \mathrm{e}^{-2NU_\mu\left(x\right)}\dd\alpha_K\left(x\right),
 \qquad
 \limsup_{N\to\infty}\frac1N\log I_N\leq0
 \quad\text{almost surely}.
\]
This is the hypothesis of Lemma~\ref{LemIntervalNorm}, which is a known
weighted extremality criterion from \cite{BloomWeighted,MhaskarSaff}.
For the weight \(\mathrm{e}^{-U_\mu}\), the corresponding equilibrium measure is
\(\mu\). The criterion allows complex zeros and yields
\(\nu_{q_N}\to\mu\), proving Proposition~\ref{PropHermitianEndpoint}.
The target matrix law is again invoked separately for the corresponding
assertion of Theorem~\ref{ThmBoundary} and then restoring the affine parameters
completes both theorems. In general, all almost-sure estimates use summable marginal
bounds, so independence across matrix sizes is unnecessary.

\begin{center}
\begin{minipage}{\linewidth}
\centering
\setlength{\unitlength}{1mm}
\begin{picture}(133,91)
\put(0,87){\makebox(64,4){\color{PlanarBlue}\small\textbf{Planar target}}}
\put(69,87){\makebox(64,4){\color{BoundaryTeal}\small\textbf{Boundary target}}}
\put(0,67){\color{PlanarBlue}\framebox(64,18){\parbox{60mm}{\centering\footnotesize
Initial elliptic law and kernel\\
\(\Downarrow\)\\Shifted norm rates for \(p_N\)\\
(Proposition~\ref{PropShiftedExponent})}}}
\put(69,67){\color{BoundaryTeal}\framebox(64,18){\parbox{60mm}{\centering\footnotesize
Target mean law, density and concentration\\
\(\Downarrow\)\\Uniform determinant bound\\
(Lemma~\ref{LemDeterminantBound})}}}
\put(32,64){\color{PlanarBlue}\makebox(0,0){\(\downarrow\)}}
\put(101,64){\color{BoundaryTeal}\makebox(0,0){\(\downarrow\)}}
\put(0,44){\color{PlanarBlue}\framebox(64,17){\parbox{60mm}{\centering\footnotesize
\textbf{Known pathwise identity}\\
\(\|q_N\|_{\gamma_1}^2=\|p_N\|_{\gamma_0}^2\)\\
Driver--Hall--Kemp (Proposition~\ref{PropNormTransfer})}}}
\put(69,44){\color{BoundaryTeal}\framebox(64,17){\parbox{60mm}{\centering\footnotesize
\textbf{Known second-moment identity}\\
\(\E[|q_N(z)|^2]=\E[|\det(z\mathbf I-\mathbf B_N)|^2]\)\\
Hall--Ho (Proposition~\ref{PropMatrixDeformation})}}}
\put(32,41){\color{PlanarBlue}\makebox(0,0){\(\downarrow\)}}
\put(101,41){\color{BoundaryTeal}\makebox(0,0){\(\downarrow\)}}
\put(0,23){\color{PlanarBlue}\framebox(64,15){\parbox{60mm}{\centering\footnotesize
Transferred norm rates for \(q_N\)\\
and global growth bound}}}
\put(69,23){\color{BoundaryTeal}\framebox(64,15){\parbox{60mm}{\centering\footnotesize
Expected weighted norm bound\\
\(\Downarrow\)\\
Almost-sure weighted norm bound}}}
\put(32,20){\color{PlanarBlue}\makebox(0,0){\(\downarrow\)}}
\put(101,20){\color{BoundaryTeal}\makebox(0,0){\(\downarrow\)}}
\put(0,1){\color{PlanarBlue}\framebox(64,16){\parbox{60mm}{\centering\footnotesize
Interior potential equality\\
\(\Downarrow\)\\Planar zero limit\\
(Proposition~\ref{PropNormalisedPlanar})}}}
\put(69,1){\color{BoundaryTeal}\framebox(64,16){\parbox{60mm}{\centering\footnotesize
Known weighted extremality\\
(Lemma~\ref{LemIntervalNorm})\\
\(\Downarrow\)\\Boundary zero limit (Proposition~\ref{PropHermitianEndpoint})}}}
\end{picture}

\par\smallskip
\refstepcounter{figure}\label{FigProofStrategy}
{\footnotesize Figure~\thefigure. From matrix estimates to heat-evolved zeros.
Here \(\gamma_j=\gamma_{N;v,c_j,a}\) and
\(\|P\|_{\gamma_j}\) denotes its \(\mathrm L^2(\gamma_j)\) norm.
The planar argument uses the initial matrix law; the boundary argument
uses the target matrix law. The exact identities transfer the stronger
estimates needed for the respective zero-convergence arguments.}
\end{minipage}
\end{center}

\paragraph{AI disclosure} In the course of the research presented here I have been using AI tools (mainly ChatGPT 5.5, 5.6 Sol, 6) for ideation, technical help, including the simulations, editing and checking the proofs and general editing and proofreading of the manuscript at the level of a co-author. All mathematical arguments have been carefully verified. Any remaining mistakes are my own responsibility.

\section{Gaussian heat identities}\label{SectionIdentities}

We recall the Gaussian identities used to transfer polynomial norms
and second moments under the heat flow.  We then record the subharmonic
compactness principle that turns the resulting potential estimates
into convergence of zero measures in the planar proof.

\subsection{Gaussian models}\label{SectionGaussianRealization}

For \(S>0\) and \(|c|\leq S\), we use the Gaussian matrix
\(\mathbf Z_N(S,c)\) defined by \eqref{EqHallCovariance}, with \(s=S\).
Writing \(c=|c|\mathrm{e}^{2\mathrm{i}\theta}\), with \(\theta=0\) if
\(c=0\), gives the realisation
\begin{equation}\label{EqGaussianRealisation}
 \mathbf{Z}_N\left(S,c\right)\overset{\mathrm{d}}{=}
 \mathrm{e}^{\mathrm{i}\theta}
 \left(\sqrt{\frac{S+\left|c\right|}{2}}\mathbf{H}_{1,N}
       +\mathrm{i}\sqrt{\frac{S-\left|c\right|}{2}}\mathbf{H}_{2,N}\right),
\end{equation}
where the two matrices are independent GUE matrices normalised by
\(\E[(\mathbf{H}_{r,N})_{ij}(\mathbf{H}_{r,N})_{kl}]
=N^{-1}\delta_{il}\delta_{jk}\).
In particular, \(c=S\) gives GUE of variance \(S\), while \(c=0\)
gives complex Ginibre with entry variance \(S/N\).

\subsection{Shifted scalar norms}\label{SectionScalarNorms}

We choose scalar Gaussian weights whose holomorphic covariance matches
the heat parameter.  For \(v>|c|\) and \(a\in\C\), define
\begin{equation}\label{EqComplexQuadratic}
 Q_{v,c,a}\left(z\right)\defeq
 \frac{v\left|z-a\right|^2-\Re\!\left(\overline{c}\left(z-a\right)^2\right)}
      {v^2-\left|c\right|^2},
\end{equation}
and the Gaussian probability measure
\begin{equation}\label{EqScalarGaussian}
 \dd\gamma_{N;v,c,a}\left(z\right)
 \defeq\frac{N}{\pi\sqrt{v^2-\left|c\right|^2}}
   \exp\!\left(-NQ_{v,c,a}\left(z\right)\right)\dd A\left(z\right).
\end{equation}
If \(\mathsf{Z}\) has this law and \(\mathsf{W}=\mathsf{Z}-a\), then
\begin{equation}\label{EqScalarMoments}
 \E\left[\mathsf{W}^2\right]=\frac{c}{N},
 \qquad \E\left[\left|\mathsf{W}\right|^2\right]=\frac{v}{N}.
\end{equation}
The monic heat polynomials
\begin{equation}\label{EqHeatPolynomials}
 h_{k,c,N}\left(z\right)\defeq\exp\!\left(-\frac{c}{2N}\partial_z^2\right)z^k
\end{equation}
are scaled Hermite polynomials.  Their orthogonality follows from
\cite[Equation~(4.6)]{AkemannVernizzi}.  After translation by \(a\),
rotation by \(-\theta\), and scaling by \(\sqrt{N/v}\), the formula is
\begin{equation}\label{EqComplexOrthogonality}
 \int_{\C}h_{k,c,N}\left(z-a\right)\overline{h_{\ell,c,N}\left(z-a\right)}
       \dd\gamma_{N;v,c,a}\left(z\right)
 =\delta_{k\ell}k!\left(\frac{v}{N}\right)^k.
\end{equation}
Here \(c=|c|\mathrm{e}^{2\mathrm{i}\theta}\), and the ellipticity
parameter in that reference is \(|c|/v\); the case \(c=0\) is the
monomial orthogonality for the circular Gaussian measure.

The next identity transfers the Gaussian norms used in the planar
proof.  It is the polynomial, one-dimensional case of complex-time
Segal--Bargmann unitarity, composed at two times and translated by \(a\).
\begin{prop}[{\cite[Theorem~1.6 and Section~1.4]{DriverHallKemp}}]
\label{PropNormTransfer}
Let \(N\geq1\), \(c_0,c_1,a\in\C\), and
\(v>\max(|c_0|,|c_1|)\).  For every polynomial \(p\), set
\(q=\exp(-(c_1-c_0)\partial_z^2/(2N))p\).  Then
\begin{equation}\label{EqNormTransfer}
 \int_{\C}\left|q\left(z\right)\right|^2\dd\gamma_{N;v,c_1,a}\left(z\right)
 =\int_{\C}\left|p\left(z\right)\right|^2\dd\gamma_{N;v,c_0,a}\left(z\right).
\end{equation}
\end{prop}

To match the parameters of \cite{DriverHallKemp}, take
\(s=v/N\) and \(\tau_j=(v-c_j)/N\).  Their admissibility condition
\(|\tau_j-s|<s\) is \(|c_j|<v\), and their range measure is
\(\gamma_{N;v,c_j,0}\).  On polynomials their transform is
\(B_{\tau_j}=\exp(\tau_j\partial_z^2/2)\), so
\(B_{\tau_1}B_{\tau_0}^{-1}=\exp(-(c_1-c_0)\partial_z^2/(2N))\);
applying this to \(p(a+\cdot)\) gives \eqref{EqNormTransfer}.

\subsection{Matrix covariance deformation}\label{SectionMatrixIdentity}

The boundary case proof uses Hall--Ho's second-moment deformation to transfer
a determinant estimate to the heat-evolved polynomial.  Their theorem
already includes the degenerate covariances on \(|c|=S\).
\begin{prop}[{\cite[Theorem~2.7, Equation~(2.11)]{HallHo}}]
\label{PropMatrixDeformation}
Let \(N\geq1\), \(S>0\), \(|c_0|,|c_1|\leq S\), and let
\(\mathbf{A}\) be a deterministic complex \(N\times N\) matrix.  Put
\[
 p_j\left(z\right)\defeq\det\left(z\mathbf{I}-\mathbf{A}-\mathbf{Z}_N\left(S,c_j\right)\right),
 \qquad
 q\defeq\exp\!\left(-\frac{c_1-c_0}{2N}\partial_z^2\right)p_0.
\]
For every \(z\in\C\),
\begin{equation}\label{EqMatrixDeformation}
 \E\left[\left|q\left(z\right)\right|^2\right]
 =\E\left[\left|p_1\left(z\right)\right|^2\right].
\end{equation}
\end{prop}

The parameter correspondence is \(s=S\) and \(\tau_j=S-c_j\), so that
\(|\tau_j-s|\leq s\) is exactly \(|c_j|\leq S\).

\subsection{Subharmonic compactness}\label{SectionCompactness}

For a compactly supported finite positive measure \(\rho\) on \(\C\),
write
\[
 U_\rho\left(z\right)\defeq\int_{\C}\log\left|z-w\right|\dd\rho\left(w\right).
\]
The potential is locally integrable, possibly taking the value
\(-\infty\), and \(\Delta U_\rho=2\pi\rho\) in distributions, with
\(\Delta=\partial_x^2+\partial_y^2\).  In particular, a monic polynomial
\(P_N\) of degree \(N\) satisfies
\begin{equation}\label{EqPolynomialLaplacian}
 u_N\defeq\frac1N\log\left|P_N\right|=U_{\nu_{P_N}},
 \qquad \frac1{2\pi}\Delta u_N=\nu_{P_N}.
\end{equation}

The following standard principle (Hartogs' lemma) supplies convergent subsequences of
potentials in the planar proof and transfers upper bounds to their
limits \cite[Theorem~1.6.4]{VuPluripotential}.
\begin{lem}\label{LemSubharmonicCompactness}
Let \(\Omega\subset\C\) be a domain, and let \((u_n)\) be subharmonic
functions locally uniformly bounded above.  Either \(u_n\to-\infty\)
uniformly on compact subsets, or a subsequence converges in
\(\mathrm{L}^1_{\mathrm{loc}}(\Omega)\) to a subharmonic function.
If \(u_n\to u\) in \(\mathrm{L}^1_{\mathrm{loc}}(\Omega)\), then for every
compact \(K\subset\Omega\) and continuous real-valued \(f\) on \(K\),
\begin{equation}\label{EqHartogs}
 \limsup_{n\to\infty}\sup_K\left(u_n-f\right)\leq\sup_K\left(u-f\right).
\end{equation}
\end{lem}

For monic polynomials the disk-mean identity gives
\begin{equation}\label{EqMonicDiskMean}
 \frac1{\pi R^2}\int_{\mathscr{D}\left(0,R\right)}\frac1N\log\left|P_N\left(z\right)\right|\dd A\left(z\right)
 \geq\log R-\frac12,\qquad R>0.
\end{equation}
Indeed, the disk average of \(\log|z-w|\) is
\(\log R-\tfrac12+|w|^2/(2R^2)\) for \(|w|\leq R\), and
\(\log|w|\) otherwise.  This excludes the \(-\infty\) alternative for
every subsequence of a locally upper-bounded family
\((N^{-1}\log|P_N|)\) on \(\C\).
Along an \(\mathrm{L}^1_{\mathrm{loc}}\)-convergent subsequence with limit \(u\),
the zero measures converge distributionally, and hence vaguely, to
\((2\pi)^{-1}\Delta u\), whose mass is at most one.  When this limiting
measure has mass one, the convergence is weak.

\section{Fixed-parameter elliptic estimates}
\label{SectionPlanarPreliminaries}

We first work at mixed covariance one.  Fix \(c\in\C\) with \(|c|<1\),
and write
\[
 \begin{gathered}
 \mathbf{X}_{N,c}\defeq\mathbf{Z}_N\left(1,c\right),\qquad
 p_{N,c}\left(z\right)\defeq\det\left(z\mathbf{I}-\mathbf{X}_{N,c}\right),\\
 \mathscr{E}_c\defeq\mathscr{E}_{1,c},\qquad \mu_c\defeq\mu_{1,c}.
 \end{gathered}
\]
We recall the classical elliptic eigenvalue kernel and logarithmic
potential, and explain how the elliptic law holds for arbitrary
couplings across \(N\).

\subsection{The elliptic eigenvalue process}
\label{SectionEllipticProcess}

The orthogonal-polynomial representation of the elliptic eigenvalue
process supplies the concentration used below and the kernel underlying
the characteristic-polynomial second moment in
Lemma~\ref{LemEvaluationKernel}.
\begin{lem}[{\cite[Equations~(2.4), (2.9), (3.12), and (4.5)--(4.7)]{AkemannVernizzi}}]
\label{LemEllipticDPP}
Let \(\Delta(\lambda)=\prod_{i<j}(\lambda_j-\lambda_i)\).  The eigenvalues
of \(\mathbf{X}_{N,c}\), in uniformly random order, have joint density
\begin{equation}\label{EqEigenvalueDensity}
 \frac{1}{\mathcal{Z}_{N,c}}\left|\Delta\left(\lambda\right)\right|^2
 \prod_{j=1}^N\dd\gamma_{N;1,c,0}\left(\lambda_j\right),
 \qquad
 \mathcal{Z}_{N,c}\defeq N!\prod_{k=0}^{N-1}\frac{k!}{N^k}.
\end{equation}
Relative to \(\gamma_{N;1,c,0}\), they form a determinantal point process
with kernel
\begin{equation}\label{EqEllipticKernel}
 K_{N,c}\left(z,w\right)\defeq\sum_{k=0}^{N-1}
 \frac{h_{k,c,N}\left(z\right)\overline{h_{k,c,N}\left(w\right)}}{\kappa_{k,N}},
 \qquad \kappa_{k,N}\defeq\frac{k!}{N^k}.
\end{equation}
This kernel is the orthogonal projection in
\(\mathrm{L}^2(\gamma_{N;1,c,0})\) onto the polynomials of degree less
than \(N\).
\end{lem}

The correspondence with \cite{AkemannVernizzi} is
\(c=r\mathrm{e}^{2\mathrm{i}\theta}\) and
\(\xi=\sqrt N\mathrm{e}^{-\mathrm{i}\theta}z\), with their
ellipticity parameter equal to \(r\).  Normalising the weight to a
probability measure and rescaling the monic polynomials gives
\(\kappa_{k,N}=k!/N^k\); their circular Gaussian formulas cover \(r=0\).

Mehler's formula bounds every finite diagonal Hermite kernel by an
explicit Gaussian.  We record it in our normalisation for the
evaluation estimates in Section~\ref{SectionPlanarProof}.
\begin{lem}[{\cite[Equation~18.18.28]{DLMF}}]\label{LemComplexMehler}
For \(z\in\C\) and \(0\leq\ell\leq1\),
\begin{equation}\label{EqComplexMehler}
 \sum_{k=0}^{\infty}\ell^k\frac{N^k}{k!}\left|h_{k,c,N}\left(z\right)\right|^2
 =\frac{1}{\sqrt{1-\left|c\right|^2\ell^2}}
 \exp\!\left[
 N\ell\frac{\left|z\right|^2-\ell\Re\left(\overline{c}z^2\right)}{1-\left|c\right|^2\ell^2}
 \right].
\end{equation}
\end{lem}

For \(c=r\mathrm{e}^{2\mathrm{i}\theta}\), \(r>0\), the relation to
the physicists' Hermite polynomials is
\[
 h_{k,c,N}\left(\mathrm{e}^{\mathrm{i}\theta}\zeta\right)
 =\mathrm{e}^{\mathrm{i}k\theta}\left(\frac{r}{2N}\right)^{k/2}
 H_k\!\left(\sqrt{\frac{N}{2r}}\,\zeta\right).
\]
The cited formula applies to conjugate complex arguments by analytic
continuation, with parameter \(r\ell<1\); the case \(r=0\) is the
exponential series.

The elliptic law is classical.  We retain the short concentration
argument that makes its almost-sure conclusion valid for every coupling
of the prescribed matrix marginals.
\begin{lem}\label{LemAlmostSureEllipticLaw}
For every coupling of the matrices \(\mathbf{X}_{N,c}\) across \(N\),
\begin{equation}\label{EqAlmostSureEllipticLaw}
 L_{\mathbf{X}_{N,c}}\overset{\mathrm{w}}{\longrightarrow}\mu_c
 \qquad\text{almost surely}.
\end{equation}
\end{lem}

\begin{proof}
After rotation to \(c=r=|c|\), the Gaussian entries satisfy
\cite[Definition~1.3 and Theorem~1.8]{NguyenORourke}, with
off-diagonal atom parameters \(\mu=1/2\) and \(\rho=r\).
Applying that theorem to corners of one infinite Gaussian array and
then bounded convergence gives
\(\E[\int_{\C} f\,\dd L_{\mathbf{X}_{N,c}}]\to\int_{\C} f\,\dd\mu_c\)
for every bounded continuous \(f\); this assertion depends only on
the individual matrix laws.
For bounded real \(f\ne0\), the functional
\(\eta\mapsto(2\|f\|_\infty)^{-1}\int_{\C} f\,\dd\eta\) is
Lipschitz with constant at most one for the total variation metric on
finite counting measures.  Lemma~\ref{LemEllipticDPP} and
\cite[Theorem~3.5]{PemantlePeresConcentration}, applied to this
functional and its negative, therefore give, for each \(\varepsilon>0\),
\[
 \Prob\left\{
 \left|\int_{\C}f\,\dd L_{\mathbf{X}_{N,c}}
 -\E\left[\int_{\C}f\,\dd L_{\mathbf{X}_{N,c}}\right]\right|>\varepsilon
 \right\}\leq6\mathrm{e}^{-b_{f,\varepsilon}N},
 \qquad b_{f,\varepsilon}>0.
\]
The zero function is immediate.  The first Borel--Cantelli lemma,
applied to a countable determining family in \(\mathrm{C}_c(\C;\R)\)
and positive rational \(\varepsilon\), gives almost-sure vague
convergence to \(\mu_c\).  Its mass is one, so the convergence is weak. Note that,
no independence across \(N\) is required in the argument.
\end{proof}

\subsection{The interior ellipse potential}

Write
\[
 U_c\left(z\right)\defeq\int_{\C}\log\left|z-w\right|\dd\mu_c\left(w\right),\qquad
 L_c\left(u\right)\defeq u+c\overline u,
\]
and define
\begin{equation}\label{EqEllipticObstacle}
 V_c\left(z\right)\defeq\frac12Q_{1,c,0}\left(z\right)-\frac12.
\end{equation}

Inside the ellipse, its potential is a quadratic function that
determines the shifted norm exponents in Section~\ref{SectionPlanarProof}.
The following formula is the interior part of the classical ellipse
potential.
\begin{lem}[{\cite[Lemma~4.1]{HallHoJalowyKabluchkoZeros}}]
\label{LemEllipticPotential}
The potential \(U_c\) is continuous on \(\C\), and
\begin{equation}\label{EqInteriorEllipsePotential}
 U_c\left(z\right)=V_c\left(z\right)\qquad\left(z\in\mathscr{E}_c\right).
\end{equation}
\end{lem}

For \(c=r\mathrm{e}^{2\mathrm{i}\theta}\), rotate the formula in the
reference by \(\theta\); its ellipse has semiaxes \(1+r\) and \(1-r\).
The disk case is the limit \(r=0\), and continuity also follows from
the bounded compactly supported density of \(\mu_c\).

\section{Proof of the planar theorem}
\label{SectionPlanarProof}

Fix \(c_0,c_1\in\C\) with \(|c_0|,|c_1|<1\), and set
\begin{equation}\label{EqPlanarPQ}
 p_N\defeq p_{N,c_0},\qquad
 q_N\defeq\exp\!\left(-\frac{c_1-c_0}{2N}\partial_z^2\right)p_N.
\end{equation}
The following proposition gives the zero limit at mixed covariance one.
Adding the independent Gaussian source and restoring scale and translation
will then prove Theorem~\ref{ThmPlanar}.

\begin{prop}\label{PropNormalisedPlanar}
For every coupling of \((\mathbf{X}_{N,c_0})_{N\geq1}\) with the
prescribed elliptic Gaussian marginals,
\[
 \nu_{q_N}\overset{\mathrm{w}}{\longrightarrow}\mu_{c_1}
 \qquad\text{almost surely}.
\]
\end{prop}

Choose once and for all
\begin{equation}\label{EqAuxiliaryVariance}
 \max\left\{\left|c_0\right|,\left|c_1\right|\right\}<v<1.
\end{equation}

\subsection{Quadratic completion and polynomial bounds}

The next quadratic completion locates the unique point selected by each
shifted Gaussian norm. Its strict maximum will force equality with the
target potential at that point.

For \(|c|<v<1\), put
\[
 B_{v,c}\left(z\right)\defeq Q_{v,c,0}\left(z\right)-Q_{1,c,0}\left(z\right).
\]
Writing \(c=r\mathrm{e}^{2\mathrm{i}\theta}\) with \(r=|c|\), the
coefficients of this real quadratic form in the rotated coordinates are
\[
 \frac1{v\pm r}-\frac1{1\pm r}
 =\frac{1-v}{\left(v\pm r\right)\left(1\pm r\right)}>0.
\]
Thus \(B_{v,c}\) is positive definite. Completing the two squares and
integrating the resulting centred Gaussian give, for every \(a\in\C\),
\begin{equation}\label{EqQuadraticCompletion}
 Q_{1,c,0}\left(z\right)-Q_{v,c,a}\left(z\right)
 =\frac{\left|a\right|^2}{1-v}
 -B_{v,c}\!\left(z-\frac{a+c\overline a}{1-v}\right),
\end{equation}
and
\begin{equation}\label{EqGaussianCompletion}
 \int_{\C}\mathrm{e}^{NQ_{1,c,0}\left(z\right)}\dd\gamma_{N;v,c,a}\left(z\right)
 =\frac{\sqrt{1-\left|c\right|^2}}{1-v}
   \exp\!\left(\frac{N\left|a\right|^2}{1-v}\right).
\end{equation}
In particular, if \(a=(1-v)u\), then
\begin{equation}\label{EqStrictQuadratic}
 2V_c\left(z\right)-Q_{v,c,\left(1-v\right)u}\left(z\right)
 =-1+\left(1-v\right)\left|u\right|^2-B_{v,c}\left(z-L_c\left(u\right)\right).
\end{equation}

We next recall the characteristic-polynomial kernel identity, which is
the \(K=L=1\) case of \cite[equations~(2.11)--(2.12)]{AkemannVernizzi}.
Combined with Mehler's formula, it bounds both the expected squared
determinant and pointwise polynomial evaluation.

\begin{lem}\label{LemEvaluationKernel}
For \(|c|<1\),
\begin{equation}\label{EqChristoffelIdentity}
 \E\left[\left|p_{N,c}\left(z\right)\right|^2\right]
 =\frac{N!}{N^N}\sum_{k=0}^N\frac{N^k}{k!}\left|h_{k,c,N}\left(z\right)\right|^2.
\end{equation}
Consequently,
\begin{equation}\label{EqChristoffelUpper}
 \E\left[\left|p_{N,c}\left(z\right)\right|^2\right]
 \leq\frac{N!}{N^N\sqrt{1-\left|c\right|^2}}\,
           \mathrm{e}^{NQ_{1,c,0}\left(z\right)}.
\end{equation}
Every polynomial \(P\) of degree at most \(N\) also satisfies
\begin{equation}\label{EqEvaluationBound}
 \left|P\left(z\right)\right|^2\leq
 \frac{\left\|P\right\|_{\mathrm{L}^2\left(\gamma_{N;1,c,0}\right)}^2}{\sqrt{1-\left|c\right|^2}}
 \mathrm{e}^{NQ_{1,c,0}\left(z\right)}.
\end{equation}
\end{lem}

\begin{proof}
Apply \cite[equations~(2.11)--(2.12)]{AkemannVernizzi} at fixed \(N\)
with \(K=L=1\) and both evaluation arguments equal to \(z\). By
\eqref{EqComplexOrthogonality}, the degree-\(N\) monic squared norm is
\(N!/N^N\), and the orthonormal polynomials are
\(N^{k/2}h_{k,c,N}/\sqrt{k!}\). The resulting kernel sum runs through
degree \(N\), giving \eqref{EqChristoffelIdentity}. Bounding this sum
by Mehler's formula \eqref{EqComplexMehler} at \(\ell=1\) gives
\eqref{EqChristoffelUpper}. Finally, Cauchy--Schwarz in the same
orthonormal basis bounds \(\left|P(z)\right|^2\) by its squared norm
times that truncated kernel. The same Mehler bound proves
\eqref{EqEvaluationBound}.
\end{proof}

\subsection{Shifted norms and global growth}
\label{SectionShiftedNorms}

The following proposition computes the exponential rate of each shifted
norm of the initial characteristic polynomial. Norm transfer will give
the same rate for the heat-evolved polynomial at the target covariance.

\begin{prop}\label{PropShiftedExponent}
For \(u\in\D\), put \(a=(1-v)u\) and
\[
 J_N\left(u\right)\defeq\int_{\C}\left|p_N\left(z\right)\right|^2\dd\gamma_{N;v,c_0,a}\left(z\right).
\]
Then
\begin{equation}\label{EqShiftedExponent}
 \frac1N\log J_N\left(u\right)\longrightarrow-1+\left(1-v\right)\left|u\right|^2
 \qquad\text{almost surely}.
\end{equation}
The conclusion holds simultaneously for every prescribed countable
subset of \(\D\).
\end{prop}

\begin{proof}
Combining \eqref{EqChristoffelUpper} and \eqref{EqGaussianCompletion}
gives
\[
 \E\left[J_N\left(u\right)\right]\leq\frac{N!}{N^N\left(1-v\right)}
                      \mathrm{e}^{N\left(1-v\right)\left|u\right|^2}.
\]
Markov's inequality with factor \(\mathrm{e}^{\sqrt N}\), the first
Borel--Cantelli lemma, and Stirling's formula yield
\begin{equation}\label{EqShiftedUpper}
 \limsup_N\frac1N\log J_N\left(u\right)\leq-1+\left(1-v\right)\left|u\right|^2
 \qquad\text{almost surely}.
\end{equation}

For the lower bound, work on the event in
Lemma~\ref{LemAlmostSureEllipticLaw}, and put
\(z_0=L_{c_0}(u)\).  For a fixed disk
\(D=\mathscr{D}(z_0,r)\), the averaged logarithmic kernel is
\[
 F_D\left(w\right)\defeq\frac1{\pi r^2}\int_D\log\left|z-w\right|\dd A\left(z\right).
\]
This is the potential of the uniform probability measure on \(D\),
so it is continuous and bounded below on \(\C\). Factorisation of \(p_N\)
and Jensen's inequality on \(D\) imply
\[
 J_N\left(u\right)\geq\frac{Nr^2}{\sqrt{v^2-\left|c_0\right|^2}}
 \exp\!\left(
 N\left[2\int_{\C} F_D\dd L_{\mathbf{X}_{N,c_0}}
               -\sup_D Q_{v,c_0,a}\right]\right).
\]
Portmanteau applies to \(F_D\), since it is bounded below, and gives
\[
 \liminf_N\frac1N\log J_N\left(u\right)
 \geq\frac2{\pi r^2}\int_D U_{c_0}\left(z\right)\dd A\left(z\right)
       -\sup_D Q_{v,c_0,a}.
\]
Letting \(r\downarrow0\) and using continuity of \(U_{c_0}\) yields
the lower bound \(2U_{c_0}(z_0)-Q_{v,c_0,a}(z_0)\).
Since \(z_0\) lies in the ellipse interior,
Lemma~\ref{LemEllipticPotential} and \eqref{EqStrictQuadratic} identify
this value as \(-1+(1-v)|u|^2\).  This proves
\eqref{EqShiftedExponent}.  The lower-bound argument is deterministic on the probability-one event
where \(L_{\mathbf X_{N,c_0}}\) converges weakly to \(\mu_{c_0}\),
supplied by Lemma~\ref{LemAlmostSureEllipticLaw}. On this same event,
the argument applies to every \(u\in\D\), without introducing further
exceptional sets. For each fixed \(u\), the Markov--Borel--Cantelli
argument gives the upper bound on a possibly \(u\)-dependent event of
probability one. Given a prescribed countable set
\(\mathscr U\subset\D\), intersect these upper-bound events with the
elliptic-law event. The intersection still has probability one, and
both bounds hold there for every \(u\in\mathscr U\), proving the
stated simultaneous convergence.
\end{proof}

The pathwise identity in Proposition~\ref{PropNormTransfer}
equates the initial and target shifted norms. Hence, on the same
probability-one event,
\begin{equation}\label{EqTargetShiftedNorm}
 \frac1N\log\int_{\C}\left|q_N\left(z\right)\right|^2
       \dd\gamma_{N;v,c_1,\left(1-v\right)u}\left(z\right)
 \longrightarrow-1+\left(1-v\right)\left|u\right|^2.
\end{equation}

We also need a bound which holds simultaneously throughout the plane.
The unshifted norm and the evaluation kernel give this bound and control
the tails in the later variational argument.

Norm transfer with \(v=1\), \(a=0\), and integration of
\eqref{EqChristoffelIdentity} using \eqref{EqComplexOrthogonality} give
\begin{equation}\label{EqExpectedGlobalNorm}
 \E\left[\left\|q_N\right\|_{\mathrm{L}^2\left(\gamma_{N;1,c_1,0}\right)}^2\right]
 =\left(N+1\right)\frac{N!}{N^N}.
\end{equation}
Write \(R_N=\|q_N\|_{\mathrm L^2(\gamma_{N;1,c_1,0})}^2\).
Markov's inequality gives
\[
 \Prob\left[R_N>\mathrm{e}^{\sqrt N}\left(N+1\right)\frac{N!}{N^N}\right]
 \leq\mathrm{e}^{-\sqrt N}.
\]
The right-hand side is summable, so Borel--Cantelli and Stirling's
formula imply, almost surely for all sufficiently large \(N\),
\[
 \log R_N\leq\sqrt N+\log\left(N+1\right)+\log\left(N!\right)-N\log N
 =-N+O\left(\sqrt N+\log N\right).
\]
Taking logarithms in the evaluation bound \eqref{EqEvaluationBound}
now gives, for every \(z\in\C\),
\[
\begin{aligned}
 \frac1N\log\left|q_N\left(z\right)\right|
 &\leq\frac12Q_{1,c_1,0}\left(z\right)
       +\frac{\log R_N}{2N}
       -\frac{\log\left(1-|c_1|^2\right)}{4N}\\
 &=V_{c_1}\left(z\right)+\frac{\log R_N+N}{2N}
       -\frac{\log\left(1-|c_1|^2\right)}{4N}.
\end{aligned}
\]
The positive part of the scalar error on the right defines a sequence
\(\varepsilon_N\geq0\) tending to zero almost surely. Thus
\begin{equation}\label{EqGlobalPlanarBound}
 \frac1N\log\left|q_N\left(z\right)\right|\leq V_{c_1}\left(z\right)+\varepsilon_N
 \qquad\left(z\in\C\right).
\end{equation}
The error depends only on \(R_N\), \(N\), and \(c_1\), so the bound
holds simultaneously throughout the plane.

\subsection{Identification by subharmonic compactness}

We now use the shifted rates to identify every possible subharmonic
limit in the ellipse interior. The argument is related to
\cite[Theorem~2.4]{BayraktarMass} and the compactness methods of
\cite{BloomLevenberg}. There one assumes convergence of weighted
polynomial mass measures in a fixed weighted setting while here the input is
the exponential rate of a family of shifted Gaussian norms. We give
the identification directly, using their distinct quadratic maximizers.

\begin{proof}[Proof of Proposition~\ref{PropNormalisedPlanar}]
Fix a countable dense set \(\mathscr{U}\subset\D\), and work on the
probability-one event where \eqref{EqTargetShiftedNorm} holds for every
\(u\in\mathscr{U}\) and \eqref{EqGlobalPlanarBound} holds.  Set
\[
 u_N\left(z\right)\defeq\frac1N\log\left|q_N\left(z\right)\right|.
\]
These subharmonic functions are locally uniformly bounded above.  Since
\(q_N\) is monic of degree \(N\), \eqref{EqMonicDiskMean} excludes
collapse to \(-\infty\).  Thus every subsequence has a further
subsequence, still indexed by \(N\), for which
\[
 u_N\longrightarrow\psi\quad\text{in }\mathrm{L}^1_{\mathrm{loc}}\left(\C\right)
\]
by Lemma~\ref{LemSubharmonicCompactness}.  The limit \(\psi\) is
subharmonic. Pass the global bound to disk averages using the
local \(\mathrm L^1\) convergence, then shrink the disks. The
submean inequality and continuity of \(V_{c_1}\) give
\begin{equation}\label{EqLimitObstacleBound}
 \psi\leq V_{c_1}\quad\text{on }\C.
\end{equation}

Fix \(u\in\mathscr{U}\), and abbreviate
\[
 a\defeq\left(1-v\right)u,\qquad Q\defeq Q_{v,c_1,a},\qquad
 M\defeq-1+\left(1-v\right)\left|u\right|^2,\qquad z_*\defeq L_{c_1}\left(u\right).
\]
The bound in Lemma~\ref{LemSubharmonicCompactness} gives, on each
compact disk \(K_R=\overline{\mathscr{D}(0,R)}\),
\[
 \limsup_N\sup_{K_R}\left(2u_N-Q\right)\leq\sup_{K_R}\left(2\psi-Q\right).
\]
Furthermore, \eqref{EqGlobalPlanarBound} and \eqref{EqStrictQuadratic}
give constants \(C\in\R\), \(\eta>0\), independent of \(N\), such that
\[
 2u_N\left(z\right)-Q\left(z\right)\leq C-\eta\left|z\right|^2+2\varepsilon_N.
\]
Put \(d_N=N/[\pi\sqrt{v^2-|c_1|^2}]\).  The integral over \(K_R\) in
\eqref{EqTargetShiftedNorm} is at most
\[
 d_N\pi R^2\exp\!\left(N\sup_{K_R}\left(2u_N-Q\right)\right),
\]
whereas its complement contributes at most
\[
 d_N\frac{\pi}{N\eta}
 \exp\!\left(N\left(C+2\varepsilon_N-\eta R^2\right)\right).
\]
The norm integral has exponential rate \(M\) by
\eqref{EqTargetShiftedNorm}. The two bounds above, the Hartogs estimate,
and \(N^{-1}\log d_N\to0\) therefore give, for fixed \(R\),
\[
 M\leq\max\left\{\sup_{K_R}\left(2\psi-Q\right),C-\eta R^2\right\}.
\]
Choose \(R\) large enough that \(C-\eta R^2<M\). It follows that
\begin{equation}\label{EqVariationalLower}
 M\leq\sup_{z\in\C}\left\{2\psi\left(z\right)-Q\left(z\right)\right\}.
\end{equation}

By \eqref{EqLimitObstacleBound} and \eqref{EqStrictQuadratic},
\[
 2\psi\left(z\right)-Q\left(z\right)\leq M-B_{v,c_1}\left(z-z_*\right).
\]
The right-hand side has the unique maximum \(M\) at \(z_*\).
The left-hand side is upper semicontinuous, is finite somewhere, and
tends to \(-\infty\) at infinity; hence it attains its supremum.
Together with \eqref{EqVariationalLower}, this forces its maximizer to
be \(z_*\) and gives
\[
 \psi\left(L_{c_1}\left(u\right)\right)=V_{c_1}\left(L_{c_1}\left(u\right)\right).
\]
The map \(L_{c_1}\) is a real-linear homeomorphism, so these points are
dense in \(\operatorname{int}\mathscr{E}_{c_1}\).  If such points
\(z_j\) tend to \(z\) in the interior, upper semicontinuity gives
\[
 \psi\left(z\right)\geq\limsup_j\psi\left(z_j\right)
 =\lim_jV_{c_1}\left(z_j\right)=V_{c_1}\left(z\right).
\]
Together with \eqref{EqLimitObstacleBound}, this proves
\(\psi=V_{c_1}\) throughout the interior.

By \eqref{EqPolynomialLaplacian}, the zero measures converge
distributionally along this subsequence to the positive measure
\(\sigma=(2\pi)^{-1}\Delta\psi\).  This measure has mass at most one:
test the convergence against compactly supported smooth functions
between zero and one, and exhaust the plane.  In the ellipse interior,
\[
 \sigma=\frac1{2\pi}\Delta V_{c_1}
 =\frac{1}{\pi\left(1-\left|c_1\right|^2\right)}\dd A.
\]
The interior already has mass one under this density, so
\(\sigma=\mu_{c_1}\) on the whole plane.  Distributional convergence of
these probability measures implies vague convergence by approximation
of continuous test functions and because the limit has mass one, the
convergence is weak.  Every subsequence has a further subsequence with
this same limit, proving
\begin{equation}\label{EqNormalisedPlanarConclusion}
 \nu_{q_N}\overset{\mathrm{w}}{\longrightarrow}\mu_{c_1}
 \qquad\text{almost surely}.
\end{equation}
\end{proof}

\subsection{Gaussian sources and affine parameters}
\label{SectionGaussianSource}

\begin{proof}[Proof of Theorem~\ref{ThmPlanar}]
Let \(\mathbf{W}_{j,N}=\mathbf{Y}_N+\mathbf Z_N(s,\kappa_j)\),
\(j=0,1\).  Independence and addition of Gaussian covariances show that
each \(\mathbf{W}_{j,N}\) has law \(\mathbf{Z}_N(S,c_j)\), with the
parameters in \eqref{EqTotalCovariances}.  Moreover,
\(c_1-c_0=\kappa_1-\kappa_0\), so the heat operator in the theorem has the
sign used in \eqref{EqPlanarPQ}.

Put \(\widehat c_j=c_j/S\) and write \(z=b_N+\sqrt S\,w\).  The
matrices \(S^{-1/2}\mathbf{W}_{j,N}\) have elliptic Gaussian marginals
\(\mathbf{Z}_N(1,\widehat c_j)\), and \(\partial_z^2=S^{-1}\partial_w^2\).
Consequently,
\[
 q_N\left(b_N+\sqrt S\,w\right)
 =S^{N/2}
 \exp\!\left(-\frac{\widehat c_1-\widehat c_0}{2N}\partial_w^2\right)
 \det\!\left(w\mathbf{I}-S^{-1/2}\mathbf{W}_{0,N}\right).
\]
Proposition~\ref{PropNormalisedPlanar} applies to this actual sequence,
regardless of its coupling across \(N\).  It follows that
\[
 \nu_{q_N}\overset{\mathrm{w}}{\longrightarrow}
 \left(w\mapsto b+\sqrt S\,w\right)_*\mu_{1,c_1/S}
 =\left(z\mapsto b+z\right)_*\mu_{S,c_1}
 \qquad\text{almost surely}.
\]
Since \(b_N\to b\), tightness and uniform convergence of the affine
maps on compact sets justify the varying translation.

Apply Lemma~\ref{LemAlmostSureEllipticLaw} to
\(S^{-1/2}\mathbf{W}_{1,N}\), followed by the same scale and translation.
This proves the asserted target-matrix limit.  Intersecting the two
probability-one events gives the joint conclusion of the theorem.
\end{proof}

\section{Determinant estimates at the boundary}
\label{SectionDeterminantBound}

We bound the expected squared characteristic polynomial of a deformed
GUE matrix.  The matrix deformation identity will transfer this estimate
to the weighted interval norm used in Section~\ref{SectionBoundaryProof}.

Fix \(s>0\).  Throughout this section, let
\(\mathbf A_N=\mathbf A_N^*\) be deterministic, with
\[
 M\defeq\sup_N\left\|\mathbf A_N\right\|<\infty,
 \qquad
 L_{\mathbf A_N}\overset{\mathrm w}{\longrightarrow}\mu_0.
\]
Let \(\mathbf H_N\) be GUE of variance \(s\), with arbitrary coupling
across \(N\), and put
\begin{equation}\label{EqDeformedGUEObjects}
 \mathbf B_N\defeq\mathbf A_N+\mathbf H_N,
 \qquad
 \overline\mu_N\defeq\E\left[L_{\mathbf B_N}\right],
 \qquad
 \mu\defeq\mu_0\boxplus\mathfrak{sc}_s.
\end{equation}
Thus, in an orthonormal real Hilbert--Schmidt basis of Hermitian matrices,
\(\mathbf H_N\) has covariance \((s/N)\mathbf I\).  To compare truncated
logarithmic potentials, we use the Wasserstein distance
\[
 \mathbb W_1\left(\rho,\eta\right)
 \defeq\inf_{\pi\in\Pi\left(\rho,\eta\right)}
 \int_{\R^2}\left|x-y\right|\dd\pi\left(x,y\right),
\]
where \(\rho,\eta\) are probability measures on \(\R\) with finite
first moments and \(\Pi(\rho,\eta)\) denotes their couplings.

The following lemma recalls the classical deformed-GUE limit and
Biane's regularity results.  We record the elementary consequences
needed for the logarithmic cutoff in the determinant estimate.
\begin{lem}\label{LemDeformedGUE}
Almost surely,
\[
 L_{\mathbf B_N}\overset{\mathrm w}{\longrightarrow}\mu.
\]
The deterministic mean measures satisfy
\[
 \mathbb W_1\left(\overline\mu_N,\mu\right)\longrightarrow0.
\]
The measure \(\mu\) is compactly supported and has a density bounded by
\(C_s=1/(\pi\sqrt{s})\).  Its logarithmic potential is continuous on
\(\C\).
\end{lem}

\begin{proof}
The almost-sure and mean weak convergence are the classical deformed-GUE
law; see \cite[Introduction and Section~2]{CapitaineDonatiMartinFeralFevrier}.
Equivalently, they follow from
\cite[Chapter~4, Section~4.2, Theorem~4]{MingoSpeicher}:
the norm bound and weak convergence of \(L_{\mathbf A_N}\) imply the
required convergence of all moments.  This theorem does not require
independence across dimensions.
The measures \(\mu_0\) and \(\mathfrak{sc}_s\) are compactly supported,
and so is \(\mu\).

The identity
\[
 \int_\R x^2\dd\overline\mu_N\left(x\right)
 =\tr_N\left(\mathbf A_N^2\right)+s\leq M^2+s
\]
gives uniform integrability of the first moments.  Together with the
mean weak convergence, it proves
\(\mathbb W_1(\overline\mu_N,\mu)\to0\).
Biane's subordination and density formulas
\cite[Section~3, Lemma~3, Proposition~2 and Corollaries~1--3]{Biane}
give absolute continuity and the bound \(C_s\).
For \(T>0\), define the continuous truncated potential
\[
 U_{\mu,T}\left(z\right)\defeq
 \int_\R\max\left\{\log\left|z-y\right|,-T\right\}\dd\mu\left(y\right).
\]
The density bound gives, uniformly in \(z\in\C\),
\begin{equation}\label{EqLogTruncationError}
 0\leq U_{\mu,T}\left(z\right)-U_\mu\left(z\right)
 \leq C_s\int_{-\mathrm{e}^{-T}}^{\mathrm{e}^{-T}}
 \log\left(\frac{\mathrm{e}^{-T}}{\left|t\right|}\right)\dd t
 =2C_s\mathrm{e}^{-T}.
\end{equation}
Here \(|z-y|\geq|\Re z-y|\) gives the same estimate for nonreal
\(z\).  Thus \(U_\mu\) is a uniform limit of continuous functions.
\end{proof}

The next estimate is uniform on the interval used for the polynomial
norm.  We use the regularized-logarithm upper-bound argument of
\cite[Section~2.2]{BenArousBourgadeMcKenna}, with a fixed cutoff followed
by a second limit, so only qualitative Wasserstein convergence is needed.
\begin{lem}\label{LemDeterminantBound}
For every compact interval \(K\subset\R\),
\begin{equation}\label{EqDeterminantBound}
 \limsup_{N\to\infty}\sup_{x\in K}
 \left\{
 \frac1{2N}\log\E\left[
 \left|\det\left(x\mathbf I-\mathbf B_N\right)\right|^2\right]
 -U_\mu\left(x\right)\right\}\leq0.
\end{equation}
\end{lem}

\begin{proof}
Fix \(T>0\), and put
\[
 f_{x,T}\left(y\right)\defeq\max\left\{\log\left|x-y\right|,-T\right\},
 \qquad
 F_{x,T}\left(\mathbf H\right)\defeq
 \operatorname{Tr}f_{x,T}\left(\mathbf A_N+\mathbf H\right).
\]
The function \(f_{x,T}\) is \(\mathrm{e}^T\)-Lipschitz, uniformly in
\(x\), and Hoffman--Wielandt \cite[Chapter~VI]{Bhatia} makes
\(F_{x,T}\) \(\sqrt N\mathrm{e}^T\)-Lipschitz in the
Hilbert--Schmidt norm.  Since \(F_{x,T}(\mathbf H_N)\) bounds
\(\log|\det(x\mathbf I-\mathbf B_N)|\), Gaussian concentration
\cite[Theorem~5.1 and equation~(5.8)]{Ledoux}, with covariance \(s/N\) and
exponential parameter \(2\), therefore gives
\begin{equation}\label{EqTruncatedDeterminant}
 \frac1{2N}\log\E\left[
 \left|\det\left(x\mathbf I-\mathbf B_N\right)\right|^2\right]
 \leq\int_\R f_{x,T}\left(y\right)\dd\overline\mu_N\left(y\right)
 +\frac{s\mathrm{e}^{2T}}N.
\end{equation}
The Kantorovich--Rubinstein inequality applies after subtracting
\(f_{x,T}(0)\), since the measures have finite first moments.
Combining it with \eqref{EqLogTruncationError} gives
\[
 \frac1{2N}\log\E\left[
 \left|\det\left(x\mathbf I-\mathbf B_N\right)\right|^2\right]
 -U_\mu\left(x\right)
 \leq\mathrm{e}^T\mathbb W_1\left(\overline\mu_N,\mu\right)
       +2C_s\mathrm{e}^{-T}+\frac{s\mathrm{e}^{2T}}N.
\]
The right-hand side is independent of \(x\).  First let \(N\to\infty\),
using Lemma~\ref{LemDeformedGUE}, and then let \(T\to\infty\).
\end{proof}

\section{Proof of the boundary theorem}\label{SectionBoundaryProof}

We use weighted asymptotic extremality to convert a subexponential
interval norm into convergence of zero measures.  The determinant
estimate in Section~\ref{SectionDeterminantBound} supplies this norm
bound for the heat-evolved characteristic polynomial.

The following criterion combines the weighted Bernstein--Markov
property of \cite[Theorem~3.2]{BloomWeighted} with
\cite[Theorem~2.3(b)]{MhaskarSaff}.  We give the weight and equilibrium
measure identification needed to apply these results.
\begin{lem}\label{LemIntervalNorm}
Let \(K=[m-L,m+L]\), where \(L>0\), and let
\[
 \dd\alpha_K\left(x\right)
 \defeq\frac{\mathbf1_K\left(x\right)}
 {\pi\sqrt{L^2-\left(x-m\right)^2}}\dd x.
\]
Let \(\mu\) be a probability measure supported on \(K\), with
\(U_\mu\) continuous on \(\C\).  If \(P_N\) is monic of degree \(N\)
and
\begin{equation}\label{EqIntervalNormAssumption}
 I_N\defeq\int_K\left|P_N\left(x\right)\right|^2
 \mathrm{e}^{-2NU_\mu\left(x\right)}\dd\alpha_K\left(x\right),
 \qquad
 \limsup_{N\to\infty}\frac1N\log I_N\leq0,
\end{equation}
then
\[
 \nu_{P_N}\overset{\mathrm w}{\longrightarrow}\mu.
\]
\end{lem}

\begin{proof}
Write \(\|\cdot\|_E\) for the supremum norm on a compact set \(E\),
and set \(w=\mathrm{e}^{-U_\mu}\) on \(K\).  This weight is continuous
and bounded away from zero.  Expansion in the orthonormal Chebyshev
basis of the arcsine measure, followed by Cauchy--Schwarz, gives
\[
 \left\|P\right\|_K
 \leq\sqrt{2N+1}\left\|P\right\|_{\mathrm L^2\left(\alpha_K\right)},
 \qquad \deg P\leq N.
\]
Thus \((K,\alpha_K)\) has the Bernstein--Markov property.
Since the interval \(K\) is locally regular,
\cite[Theorem~3.2]{BloomWeighted} gives its weighted version: for every
\(\varepsilon>0\), there is \(C_\varepsilon<\infty\) such that
\[
 \left\|w^NP_N\right\|_K
 \leq C_\varepsilon\left(1+\varepsilon\right)^N I_N^{1/2}.
\]
Consequently \(\limsup_N\|w^NP_N\|_K^{1/N}\leq1\).

In the negative-logarithm convention,
the external field is \(Q=-\log w=U_\mu\).  The measure \(\mu\)
has finite logarithmic energy by continuity of its potential, and
\(-U_\mu+Q=0\) throughout \(K\).  The equilibrium characterization
\cite[Section~6, equation~(6.2)]{SaffPotentialSurvey} therefore identifies
\(\mu\) as the weighted equilibrium measure, with modified Robin
constant \(F=0\).
Put \(S=\operatorname{supp}\mu\) and \(C_N=\|w^NP_N\|_S\).
Since \(S\subset\R\), its outer boundary is \(S\) itself, so the
balayage of \(\mu\) onto this boundary is \(\mu\).
The weighted Bernstein--Walsh estimate
\cite[Proposition~4.2]{MhaskarSaff}, with \(F=0\), therefore gives
\[
 \left|P_N\left(z\right)\right|
 \leq C_N\mathrm{e}^{NU_\mu\left(z\right)}
 \qquad\left(z\in\C\right).
\]
Divide by \(|z|^N\) and let \(|z|\to\infty\). Monicity and
\(U_\mu(z)-\log|z|\to0\) give \(C_N\geq1\).
Together with the preceding upper bound, this gives
\[
 1\leq\liminf_{N\to\infty}
 \left\|w^NP_N\right\|_S^{1/N}
 \leq\limsup_{N\to\infty}
 \left\|w^NP_N\right\|_S^{1/N}\leq1.
\]
The polynomial hull of the compact real set \(S\) is \(S\) itself
and has empty planar interior.  Hence the interior-zero condition of
\cite[Theorem~2.3(b)]{MhaskarSaff} is vacuous, and that theorem gives
\(\nu_{P_N}\overset{\mathrm w}{\longrightarrow}\mu\).
Its forward implication allows complex coefficients and does not
require any prior bound on the zeros.
\end{proof}

We now apply the above interval criterion to our setting.
The matrix second-moment identity supplies its weighted norm bound,
without any independence assumption across dimensions.
\begin{prop}\label{PropHermitianEndpoint}
Let \(s>0\), \(|c_0|\leq s\), and let
\(\mathbf A_N=\mathbf A_N^*\) be deterministic, with
\[
 \sup_N\left\|\mathbf A_N\right\|<\infty,
 \qquad
 L_{\mathbf A_N}\overset{\mathrm w}{\longrightarrow}\mu_0.
\]
The Gaussian matrices \(\mathbf Z_N(s,c_0)\) may be coupled arbitrarily
across \(N\).  Define
\[
 p_N\left(z\right)\defeq
 \det\left(z\mathbf I-\mathbf A_N-\mathbf Z_N\left(s,c_0\right)\right),
 \qquad
 q_N\defeq\exp\left(-\frac{s-c_0}{2N}\partial_z^2\right)p_N.
\]
Then, almost surely,
\[
 \nu_{q_N}\overset{\mathrm w}{\longrightarrow}
 \mu_0\boxplus\mathfrak{sc}_s.
\]
\end{prop}

\begin{proof}
Put \(\mu=\mu_0\boxplus\mathfrak{sc}_s\), and choose a compact interval
\(K\) containing its support.  Define \(I_N\) as in
\eqref{EqIntervalNormAssumption}, with \(P_N=q_N\).
Proposition~\ref{PropMatrixDeformation}, with \(c_1=s\), gives
\[
 \E\left[\left|q_N\left(x\right)\right|^2\right]
 =\E\left[\left|\det\left(x\mathbf I-\mathbf A_N-\mathbf H_N\right)\right|^2\right],
\]
where \(\mathbf H_N\) has the GUE law of variance \(s\).
Tonelli's theorem and Lemma~\ref{LemDeterminantBound} imply, for every
\(\varepsilon>0\) and all sufficiently large \(N\),
\[
 \E\left[I_N\right]\leq\mathrm{e}^{2\varepsilon N},
 \qquad
 \Prob\left[I_N>\mathrm{e}^{4\varepsilon N}\right]
 \leq\mathrm{e}^{-2\varepsilon N}.
\]
The first Borel--Cantelli lemma, with a countable intersection over
positive rational \(\varepsilon\), proves
\(\limsup N^{-1}\log I_N\leq0\) almost surely.
The heat operator preserves monicity, and Lemma~\ref{LemDeformedGUE}
gives continuity of \(U_\mu\).  Lemma~\ref{LemIntervalNorm} therefore
proves the proposition.
\end{proof}

Rotation reduces the general boundary parameter to the real endpoint,
and translation transports both limiting measures to the line in the
statement of Theorem~\ref{ThmBoundary}.
\begin{proof}[Proof of Theorem~\ref{ThmBoundary}]
Set
\[
 \widetilde{\mathbf Z}_N\defeq u^{-1}\mathbf Z_N\left(s,c_0\right),
 \qquad
 \widetilde c_0\defeq u^{-2}c_0,
 \qquad
 \widetilde p_N\left(\zeta\right)\defeq
 \det\left(\zeta\mathbf I-\mathbf A_N-\widetilde{\mathbf Z}_N\right).
\]
Since \(|u|=1\), the matrix \(\widetilde{\mathbf Z}_N\) has law
\(\mathbf Z_N(s,\widetilde c_0)\), with \(|\widetilde c_0|\leq s\).
The terminating heat series gives
\begin{equation}\label{EqBoundaryHeatChange}
 q_N\left(b_N+u\zeta\right)
 =u^N\exp\left(-\frac{s-\widetilde c_0}{2N}\partial_\zeta^2\right)
       \widetilde p_N\left(\zeta\right).
\end{equation}
Proposition~\ref{PropHermitianEndpoint} proves convergence of the zero
measures on the right to \(\mu_0\boxplus\mathfrak{sc}_s\).
Their pushforwards under \(\zeta\mapsto b_N+u\zeta\) are
\(\nu_{q_N}\).  Since \(b_N\to b\), tightness and uniform convergence
of these maps on compact sets give the first convergence in
\eqref{EqBoundaryLimit}.

For the target matrix, \(u^{-1}\mathbf Z_N(s,su^2)\) has the GUE
law of variance \(s\).  The classical convergence recalled in
Lemma~\ref{LemDeformedGUE} and the same affine pushforward give the
second convergence.  Intersecting the two probability-one events
requires no independence between the initial and target sequences.
\end{proof}

\clearpage
\section{Simulations}\label{SectionSimulations}

We illustrate Theorems~\ref{ThmPlanar} and~\ref{ThmBoundary} and
Corollary~\ref{CorGinibreSemicircle} through three examples. The figures
compare the zeros of the heat-evolved polynomial with the predicted
limiting measures and, where indicated, with the corresponding
target-matrix eigenvalues.

The Gaussian matrices were sampled using the realization in
Section~\ref{SectionGaussianRealization}, and the initial roots were
obtained by numerical diagonalization. For each pair \(c_0,c_1\),
we set
\[
 q_{N,t}\defeq\exp\left(-\frac{t\left(c_1-c_0\right)}{2N}
 \partial_z^2\right)p_N,\qquad 0\leq t\leq1.
\]
While the roots are simple, differentiation of
\(q_{N,t}(z_i(t))=0\) gives
\[
 \dot z_i\left(t\right)=\frac{c_1-c_0}{N}
 \sum_{j\ne i}\frac1{z_i\left(t\right)-z_j\left(t\right)}.
\]

Figure~\ref{FigPlanarSimulation} illustrates the planar theorem with
zero source, total variance one, \(c_0=0.75\), and \(c_1=-0.75\).
At each fixed time, the predicted law is uniform on
\(\mathscr E_{1,c(t)}\), where \(c(t)=0.75-1.5t\).
The change from a horizontal to a vertical ellipse passes through the
unit disk at \(t=1/2\). The final two panels compare the polynomial
zeros with the target-matrix eigenvalues.

\begin{figure}[!htbp]
\centering
\includegraphics[width=\linewidth]{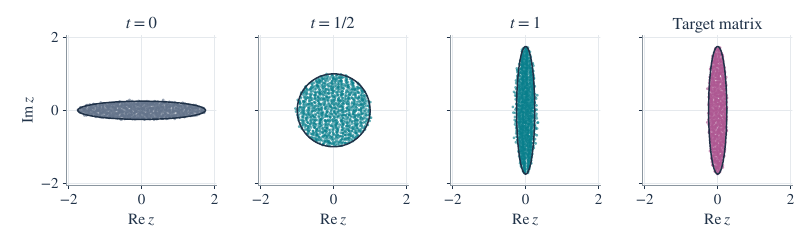}
\caption{Planar evolution with \(c_0=0.75\), \(c_1=-0.75\).
From left to right: roots at \(t=0,1/2,1\), and target-matrix
eigenvalues. Each cloud pools three independent realizations with
\(N=400\). Curves show the predicted limiting support boundaries.}
\label{FigPlanarSimulation}
\end{figure}

\clearpage
Figure~\ref{FigGinibreSimulation} takes \(p_N\) to be the characteristic
polynomial of a complex Ginibre matrix with entry variance \(1/N\),
and \(c_0=0,c_1=1\). For each fixed \(t<1\),
Theorem~\ref{ThmPlanar} predicts an ellipse with semiaxes \(1+t\)
and \(1-t\). At \(t=1\), Corollary~\ref{CorGinibreSemicircle} gives
the variance-one semicircle law. The small-\(N\) snapshots show
individual root velocities, while the larger-\(N\) panels illustrate
the endpoint distribution. The roots generally remain complex at
finite \(N\).

\begin{figure}[!htbp]
\centering
\includegraphics[width=\linewidth]{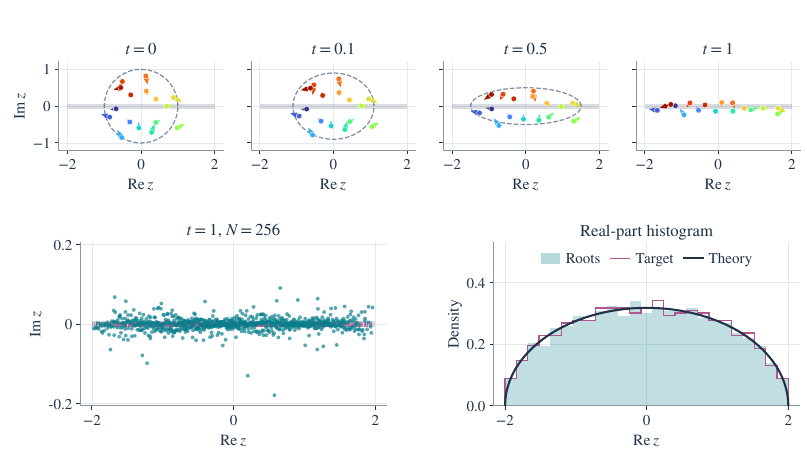}
\caption{Ginibre-to-semicircle evolution. Top: one realization with
\(N=16\) at \(t=0,0.1,0.5,1\); each arrow is \(0.3\dot z_i(t)\),
and colors identify the same root across snapshots. Dashed curves
show the fixed-time limiting ellipses; the gray segment is \([-2,2]\).
Bottom: endpoint roots and their real-part histogram, pooling three
realizations with \(N=256\), with target GUE eigenvalues and the
semicircle density for comparison.}
\label{FigGinibreSimulation}
\end{figure}

\clearpage
Finally, Figure~\ref{FigTwoPointSimulation} illustrates
Theorem~\ref{ThmBoundary} for a deterministic diagonal source with
equal numbers of entries \(-a\) and \(a\), initial Ginibre noise,
and \(c_0=0,c_1=1\). The limiting measure is
\[
 \mu_a\defeq\left(\frac12\delta_{-a}+\frac12\delta_a\right)
 \boxplus\mathfrak{sc}_1.
\]
The choices \(a=0.6,1,1.4\) give, respectively, a single support
interval, the critical case where the density vanishes at the origin,
and two separated support intervals \cite[Corollary~3]{Biane}.
These are separate parameter
choices at the same terminal time. 

\begin{figure}[!htbp]
\centering
\includegraphics[width=\linewidth]{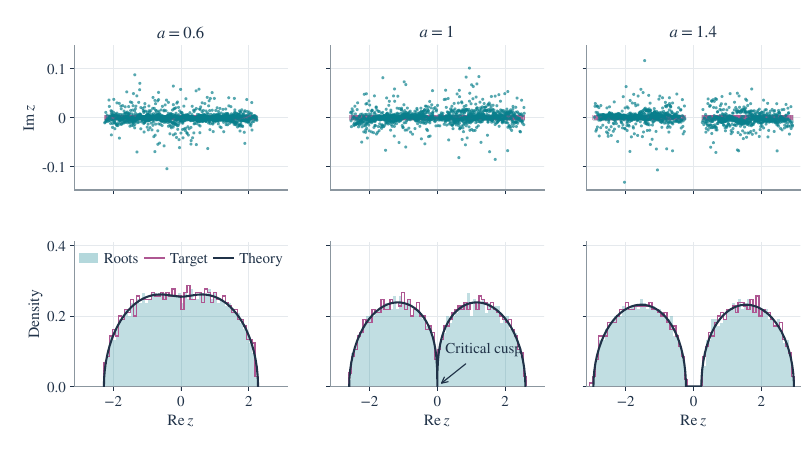}
\caption{Boundary limits with a symmetric two-point source. Columns
correspond to \(a=0.6,1,1.4\), all at \(t=1\). Each column pools three
realizations with \(N=400\). Top: actual complex roots, with expanded
imaginary scale, and target-matrix eigenvalues on the real axis.
Bottom: real-part histograms and the theoretical density of \(\mu_a\).}
\label{FigTwoPointSimulation}
\end{figure}
\clearpage

\begingroup
\raggedright
\bibliographystyle{acm}
\bibliography{HeatFlowConjRefRevision}
\endgroup

\bigskip

\noindent{\sc School of Mathematics, University of Edinburgh, James Clerk Maxwell Building, Peter Guthrie Tait Rd, Edinburgh EH9 3FD, U.K.}\newline
\href{mailto:theo.assiotis@ed.ac.uk}{\small theo.assiotis@ed.ac.uk}

\end{document}